\documentclass[letterpaper, 10 pt, conference, onecolumn]{ieeeconf}  

\IEEEoverridecommandlockouts    
\usepackage{amsmath}
\usepackage{amsfonts}
\usepackage{amssymb}
\usepackage{graphicx}
\usepackage{tikz}
\usepackage{mathtools}

\usepackage{stackrel}

\usetikzlibrary{calc}
\newcommand{\R}{\mathbb{R}}
\newcommand{\N}{\mathbb{N}}
\newcommand{\Cc}{\mathcal{C}}
\newcommand{\J}{\mathcal{J}}
\newcommand{\X}{\mathcal{X}}
\newcommand{\U}{\mathcal{U}}
\newcommand{\W}{\mathcal{W}}
\newcommand{\Ac}{\mathcal{A}}
\newcommand{\Oc}{\mathcal{O}}
\newcommand{\Lc}{\mathcal{L}}
\newcommand{\I}{\mathcal{I}}

\newcommand{\Pc}{\mathcal{P}}
\newcommand{\Sys}{\mathcal{S}}
\newcommand{\B}{\mathbb{B}}
\newcommand{\K}{\mathcal{K}}
\newcommand{\Acl}{A_{\rm cl}}
 
\newcommand{\trans}[2][]{\stackrel[#1]{#2}{{\rightsquigarrow}} }
\newcommand{\Post}{{\rm Post}}
\newcommand{\Available}{{\rm Available}}

\newcommand{\diag}{{\rm diag}}
\newcommand{\co}{{\rm co}}
\newtheorem{theorem}{Theorem}
\newtheorem{corollary}{Corollary}
\newtheorem{assumption}{Assumption}
\newtheorem{definition}{Definition}
\newtheorem{lemma}{Lemma}

\definecolor{myblue}{rgb}{0.2,0.2,0.7}
\definecolor{myred}{rgb}{0.7,0.2,0.2}
\definecolor{mygreen}{rgb}{0.2,0.6,0.2}

\newcommand{\envend}{\hfill$\triangle$}

\graphicspath{{figs/}}

\title{\LARGE \bf
	{State-feedback Abstractions for Optimal Control of Piecewise-affine Systems}
}

\author{Lucas N. Egidio$^{*}$, Thiago Alves Lima$^{*}$ and Raphaël M. Jungers$^{*}$
\thanks{R. Jungers  is a FNRS honorary Research Associate. This project has received funding from the European Research Council (ERC) under the \emph{European Union's Horizon 2020 research and innovation programme} under grant agreement No 864017 - L2C. RJ is also supported by the Innoviris Foundation and the FNRS (Chist-Era Druid-net). {The authors thank Anne-Kathrin Schmuck and Matteo Della Rossa for the fruitful discussions.} }
\thanks{$^{*}$Lucas N. Egidio, Thiago Alves Lima, and Raphaël M. Jungers are with ICTEAM Institute, Universit\'{e} Catholique de Louvain, 1348, Louvain-la-Neuve, Belgium
        {\tt\small \{lucas.egidio,thiago.alveslima, raphael.jungers\}@uclouvain.be}.%
    }%
}

\begin{document}

 \maketitle
\thispagestyle{empty}
\pagestyle{empty}

\begin{abstract}
	In this manuscript, we investigate symbolic abstractions that capture the behavior of piecewise-affine systems under input constraints and bounded external noise. This is accomplished by considering local affine feedback controllers that are jointly designed with the symbolic model, which ensures that an alternating simulation relation between the system and the abstraction holds. The resulting symbolic system is called a \textit{state-feedback abstraction} and we show that it can be deterministic even when the original piecewise-affine system is unstable and non-deterministic. One benefit of this approach is the fact that the input space need not be discretized and the symbolic-input space is reduced to a finite set of controllers. When ellipsoidal cells and affine controllers are considered, we present necessary and sufficient conditions written as a semi-definite program for the existence of a transition and a robust upper bound on the transition cost. Two examples illustrate particular aspects of the theory and its applicability.
\end{abstract}

	\section{Introduction}
	
	{Over the last years,} by leveraging formal verification methods, symbolic control techniques have provided a powerful framework for the control of complex systems under logic specifications (see the books \cite{belta2017formal} and \cite{tabuada2009verification} for surveys on this topic). {The study of cyber-physical systems under this perspective is motivated by} the increasing complexity of such dynamical systems, which intertwine more and more aspects of digital devices with real-world tasks. Some applications are, for instance, robotics~\cite{belta2007symbolic},  autonomous vehicles~\cite{borri2013decentralized}, biological systems~\cite{ghosh2004symbolic}, and temperature regulation~\cite{meyer2017compositional}. The optimal control problem, where the feedback controller must be designed while a given cost metric is minimized and user-defined specifications are respected, is the main interest of the present paper and has also received some previous attention from the research community.   
	
	For instance, in~\cite{mazo2011symbolic}, the \emph{time-optimal control problem} was studied under the symbolic approach leveraging the existence of an \emph{alternating simulation relation} between the symbolic and the real systems whereas, in~\cite{girard2012controller}, {the stronger assumption of the existence of an \emph{approximate bisimulation relation} was} adopted. Later, in~\cite{reissig2018symbolic}, a novel method {for more general \emph{undiscounted optimal control problems} was presented where the designed controller, which relies only on the symbolic (i.e., quantized) state information,} was shown to converge to the optimal one for the original system when the {adopted} discretization steps are arbitrarily small. {In the context of optimal control problems}, the authors of~\cite{legat2021abstraction} presented a branch-and-bound approach for hybrid systems with linear dynamics, which uses Q-learning to improve lower bounds and model-predictive control for obtaining upper bounds on the optimal cost function. These results were later generalized in~\cite{calbert2021alternating} for nonlinear systems under a hierarchical approach with different levels of discretizations that allow obtaining bounds on the objective cost. 

	One common point between most of these results is the fact that the input space is discretized and a \emph{growth bound} on the error between the real and the quantized state (see~\cite{reissig2016feedback, zamani2011symbolic}) is used to take into account the discretization error when computing transitions of the symbolic system. Although {this approach is numerically} efficient and {was shown to be} precise enough {for control purposes}, one main drawback {is that} in the absence of \emph{incremental stability}, it yields an over-approximation that increases the level of non-determinism in the symbolic system, as the distance between trajectories {starting close to each other can} grow over time. This non-determinism may hinder the performance of optimal path-finding algorithms such as Dijkstra and $A^*$ or even preclude the controller design whatsoever.

	In this paper, we {formalize} a novel abstraction-based approach in which the transitions are not labeled by discretized control inputs but by local {memoryless} state-feedback controllers that can ensure the determinism of the symbolic system, even when the concrete system is non-deterministic and {not incrementally stable}. This allows the robust optimal control design to be approximately solved as a shortest-path problem in a weighted digraph, instead of a hypergraph. In short, our contributions are:
	\begin{itemize} 
		\item We propose deterministic abstractions with transitions parameterized by local {affine} state-feedback controllers for non-deterministic piecewise-affine discrete-time systems{, which are shown (in Lemma~\ref{lem:S_sim_tildeS}) to be in a simulation relation with the concrete system}.
		\item We introduce a {non-conservative} numerical procedure to {decide} the existence of such state-feedback controllers ensuring a tight robust upper bound on the transition cost{, in the considered template}. This technique leverages the power of Linear Matrix Inequalities (LMIs) at the local design level {and is written as a convex optimization problem given in Corollary~\ref{coro:opt_prob_cost}}.
		\item {We state Theorem~\ref{theo:Lyapunov-like} that  generalizes Theorem~8 in~\cite{legat2021abstraction}. This allows the computation of an upper bound on the optimal cost using abstractions with overlapping cells in the state-space.}
	\end{itemize}

	Even though our design procedure for transitions demands more computational effort than other methods, such as those based on growth-bound functions~\cite{rungger2016scots}, the portion of the state-space on which these operations are carried out can be reduced by leveraging lazy frameworks (such as \cite{calbert2021alternating}) to compute the transitions only when and where needed. Moreover, the state-dependent nature of the {symbols in each} transition allow the adoption of larger discretization cells while avoiding the discretization of the input space, which can suit better systems with large state and input spaces. 

	{The approach we present in this paper can be related to previous work in the symbolic control literature. First, the resulting controller can be regarded as a piecewise-affine state-dependent control function, allowing us to characterize it as a control-driven discretization method as defined in~\cite{belta2007symbolic}. However, differently from early abstraction techniques used mainly in robotics applications such as \cite{conner2003composition,kress2009temporal}, the local controllers are synthesized prior to the path-planning problem being solved, which allows planning only over feasible trajectories of the concrete system. Some similar methodologies, but in slightly different contexts, were presented in \cite{egerstedt2003feedback}, where feedback and open-loop controllers were combined to satisfy specifications in the so-called ``motor programs'', and in \cite{fainekos2005hybrid}, in which a specific robotic system is similarly controlled, but while relying on bisimulation relations (which are in general harder to obtain than the alternating simulation relations used in our paper). More recently, the authors in \cite{nilsson2018barrier} also use local feedback controllers to ensure the existence of local barrier functions guaranteeing successful transitions between sets for continuous-time systems, and an algorithm for some robotic systems is provided.

	Also, other methodologies were successful in applying convex optimization-based techniques to synthesize abstractions, such as \cite{girard2007approximate,nilsson2017augmented,wongpiromsarn2010automatic,he2020bp}. However, our method differs from these previous works as we consider, in general, a more involved problem where transition costs are minimized in the local-control synthesis step. 
	
	\textbf{Notation:} The Minkowski sum {is} $\oplus$. By $X\succ0$ ($X\succeq0$) we denote that $X$ is a positive (semi-)definite matrix. The convex hull of $v_1,\dots,v_N$ is given as $\co\{v_1,\dots,v_N\}$. }

\section{Problem formulation}
Consider the nonlinear discrete-time system
\begin{equation}\label{eq:nonlinsys}
	x(k+1)=f\big(x(k),u(k),w(k)\big),\quad x(0)=x_0
\end{equation}
with state $x(k)\in\X\subseteq\R^{n_x}$, control input $u(k)\in\U\subseteq\R^{n_u}$ and exogenous input $w(k)\in\W\subseteq\R^{n_x}$ defined at instant $k\in\N$. The nonlinear nature of~\eqref{eq:nonlinsys} makes the process of designing feedback laws $u(k)=\kappa(x(k))$ hard to deal with, in general. The following assumption allows the derivation of a simple, yet versatile, alternative model.
\begin{assumption} \label{asmpt:bound and Lips}
	The set of exogenous inputs $\W$ is \emph{compact}  and the function $f:\X\times \U\times \W\rightarrow \X$ is \emph{Locally Lipschitz continuous}.
\end{assumption}

Under Assumption~\ref{asmpt:bound and Lips}, one can leverage the recent~\cite{sadraddini2018formal, singh2018mesh} or the classical~\cite{bemporad2005bounded,roll2004identification} identification/modeling procedures to write an alternative representation of the system~\eqref{eq:nonlinsys} through a non-deterministic piecewise-affine (PWA) system
\begin{equation}
	x(k+1) \in F(x(k),u(k)),~x(0)=x_0\label{eq:pwaffinesys}
\end{equation}
with
\begin{equation}
	F(x,u) = \{A_{\psi(x)}x+B_{\psi(x)}u+g_{\psi(x)}\}\oplus\Omega_{\psi(x)}\label{eq:fun_Fxu}
\end{equation}
where $\psi:\X\rightarrow\I_p=\{1,\dots,N_p\}$ selects one of the $N_p$ subsystems, each of which is associated to a {part} of $\X$. {We define} $\psi(x(k))$ as the index of the {part} that contains $x(k)$ at instant $k$. For each $i\in\I_p$, the set of matrices $(A_i,B_i,g_i)$ defines a \emph{nominal system} within the $i$-th {part} whereas the set $\Omega_i\subset\R^n$, with $0\in\Omega_i$,  characterizes both the exogenous input $w(k)$ and the uncertainties generated when representing the nonlinear system~\eqref{eq:nonlinsys} by the PWA system~\eqref{eq:pwaffinesys}. {As shown, for instance, in \cite[Theorem~4.6]{sadraddini2018formal}, the system~\eqref{eq:pwaffinesys} can} be built in a way that every solution to~\eqref{eq:nonlinsys} is also a solution to~\eqref{eq:pwaffinesys}. Moreover, the methods in~\cite{sadraddini2018formal} and~\cite{singh2018mesh} allow the construction of \eqref{eq:pwaffinesys} under either data-based or model-based approaches, i.e., when the non-linear model is unknown but some set of sampled trajectories are available or when the model is known and can be simulated offline.

{In this paper, we tackle an optimal control problem with a set of initial conditions $\X_0\subset\X$, a set of goal states $\X_*\subset\X$ and a set of obstacles $\mathcal{O}\subset\X$ to be avoided. The goal is to design a controller such that, for any trajectory $x(k)$ of the system~\eqref{eq:pwaffinesys} starting at a given $x_0\in \X_0$, a control sequence $u(k)\in \U$ can be generated such that there exists $K\in\N$ at which $x(K)$ enters $\X^*$ and the \emph{cost function} 
\begin{equation}
	\Cc(x,u,K) = \sum_{k=0}^{K-1} \J\big(x(k),u(k)\big)\label{eq:total_cost}
\end{equation}
is minimized, with $\J\big(x(k),u(k)\big)$ being a given \emph{stage cost}.
This is accomplished approximately by developing}
a novel type of abstract model for the PWA system~\eqref{eq:pwaffinesys} such that designing a controller (fulfilling the design specifications) for the abstract model immediately provides a controller for the PWA system {guaranteeing an upper bound on the cost~\eqref{eq:total_cost}}. 

To adopt a symbolic control approach, let us define a {\emph{transition system}} from~\eqref{eq:pwaffinesys} by the tuple $\Sys\coloneqq(\X, \U,\trans{ }_F)$ where $\trans{}_F\ \subset \X\times\U\times\X$ is the set of non-deterministic transitions defined as $\trans{}_F\ \coloneqq \{ (x,u,x') \in\X\times\U\times\X~:~ x'\in F(x,u)\}$. The inclusion $(x,u,x')\in\;\trans{}_F$ is equivalently denoted in this paper as $x\trans{u}_F x'$ which can be read as:  ``$x'\in\X$ \textit{may} be reached from $x\in\X$ by applying input $u\in\U$''. The system $\Sys$ has continuous state and input spaces, which can be discretized into cells to generate a discrete non-deterministic abstraction {(e.g., as in \cite{calbert2021alternating,rungger2016scots})}. This approach, however, is likely to hinder scalability of the control design. The methodology we propose avoids discretizing the input space $\U$ by considering instead a set of local feedback controllers $\kappa(x)$ that {can ensure deterministic transitions between discrete states}.

\section{State-feedback abstractions}
{
In this section we discuss properties of state-feedback abstractions, which are formalized in the following definition.
\begin{definition}[State-feedback abstractions]\label{def:sfabs}
	Consider a transition system given as the tuple $\Sys\coloneqq(\X, \U,\trans{ }_F)$. Consider also a transition system $\widetilde{\Sys}\coloneqq(\X_d,\K,\trans{}_{\widetilde{F}})$ where $\X_d$ is a set of cells $\xi\in\X_d$  such that $ \xi \subset \X$ and the set $\K$ contains memoryless state-feedback controllers $\kappa:\X\rightarrow\U$.
	
	 The system $\widetilde{\Sys}$ is a \emph{state-feedback abstraction} of $\Sys$ if and only if $\xi\trans{\kappa}_{\widetilde{F}} \xi'$ implies that for all $x\in \xi$ we have $F\big(x,\kappa(x)\big)\subset\xi'$. Also, the tuples in the set $\trans{\kappa}_{\widetilde{F}}$ are called \emph{state-feedback transitions} and $\Sys$ is said to be the \emph{corresponding  concrete system} of $\widetilde{\Sys}$.\envend
\end{definition}

A state-feedback abstraction can be deterministic or non-deterministic. However, in this work we focus on deterministic state-feedback abstractions as this attribute allows us to perform a symbolic control synthesis more efficiently in a digraph, rather than on a hypergraph.  Therefore, we can read $\xi\trans{\kappa}_{\widetilde{F}} \xi'$ as:  ``for all $x\in\xi$ we reach some $x'\in\xi'$ if we take $u=\kappa(x)$''. By definition, a controller $\kappa$ is available at some state $\xi\in\X_d$ in a state-feedback abstraction only if for some $\xi'\in\X_d$, when applied to the concrete system $\Sys$, it maps each $x\in\xi$ into some state $x'\in\xi'$ in one time step.}

Before introducing the simulation relation used in this paper, for {a given} set of transitions $\trans{}\ \subseteq \X\times\U\times\X$, {we} define the set-valued operators $\Post(x,\trans{},u)\coloneqq\{x'\in\X~:~x\trans{u}x'\}$  and $\Available(x,\trans{})\coloneqq\{u\in\U~:~\exists x'\in\X,~ x\trans{u}x'\}$ which enumerate, respectively, the states that may be reached from a state $x$ under input $u$ and the inputs $u$ available at $x$.

\begin{definition}[Alternating simulation relation~\cite{tabuada2009verification}]
\label{def:altsim}
Consider transition systems $\Sys_1 = (\mathcal{X}_1, \mathcal{U}_1, \trans{}_1)$ and $\Sys_2 = (\mathcal{X}_2, \mathcal{U}_2, \trans{}_2)$.
Given a relation $R \subseteq \mathcal{X}_1 \times \mathcal{X}_2$, consider the \emph{extended relation}
$R^e$ defined by the set of $(x_1, x_2, u_1, u_2)$ such that
for every $x_2' \in \Post(x_2,\trans{}_2,u_2)$,
there exists $x_1' \in \Post(x_1,\trans{}_1,u_1)$ such that $(x_1', x_2') \in R$.
If for all $(x_1, x_2) \in R$, and for all $u_1 \in \Available(x_1,\trans{}_1)$,
there exists $u_2 \in \Available(x_2,\trans{}_2)$ such that $(x_1, x_2, u_1, u_2) \in R^e$
then $R$ is an \emph{alternating simulation relation}, $R^e$ is its associated
\emph{extended alternating simulation relation} and
$\Sys_2$ is an \emph{alternating simulation} of $\Sys_1$. We denote that by  $\Sys_2\succcurlyeq_{\rm AS}{\Sys_1}$.\envend
\end{definition}

With the formal definition of an alternating simulation relation stated, let us present a lemma that relates the PWA system~\eqref{eq:pwaffinesys} and its state-feedback abstraction.
\begin{lemma}\label{lem:S_sim_tildeS}
Consider a {transition} system $\Sys\coloneqq{(\X, \U,\trans{ }_F)}$ and a corresponding state-feedback abstraction $\widetilde{\Sys}\coloneqq(\X_d,\K,\trans{}_{\widetilde{F}})$ {given as in Definition~\ref{def:sfabs}}. Then, $\Sys$ is an alternating simulation of $\widetilde{\Sys}$, that is, $\Sys\succcurlyeq_{\rm AS}\widetilde{\Sys}$.\envend
\end{lemma}
\begin{proof}
Consider the inclusion relation $R\coloneqq\{(\xi,x)\in \X_d\times \X~:~x\in\xi \}$. The extended relation $R^e$ is defined as $R^e\coloneqq\{(\xi,x,\kappa,u)\in\X_d\times\X\times\K\times\U~:~  x\in\xi,~u=\kappa(x)\}$. By construction of $\widetilde{\Sys}$, for an arbitrary tuple $(\xi,x,\kappa,u)\in R^e$ we have that  $\Post(x,\trans{}_F,u)\subset \xi'$ where $\xi'$ is the only element of $\Post(\xi,\trans{}_{\widetilde{F}},\kappa)$, showing that $(\xi',x')\in R$ for all $x'\in\Post(x,\trans{}_F,u)$. Finally, notice that for any $(\xi,x)\in R$ and for all $\kappa\in \Available(\xi,\trans{}_{\widetilde{F}})$ we have $u=\kappa(x)\in\Available(x,\trans{}_F)$ and, by definition, $(\xi,x,\kappa,u)\in R^e$, showing that $\Sys$ is an alternating simulation of $\widetilde{\Sys}$.
\end{proof}

To better illustrate the alternating simulation relation $\Sys\succcurlyeq_{\rm AS}\widetilde{\Sys}$ we depict in Figure~\ref{fig:altsim} a tuple $(\xi,x,\kappa,u)$ in the extended alternating simulation relation $R^e$. The shaded blue area denotes $\Post(x,\trans{}_F,u)$ and the shaded red area is the image of $\xi$ under the controller $u=\kappa(x)$. As expected, for all $x'\in\Post(x,\trans{}_F,u)$ we have $x'\in\xi'$. Furthermore, for any pair $(\xi,x)\in R$ (i.e., that satisfies the relation $x\in\xi$) this property holds for every available controller $\kappa\in\Available(\xi,\trans{}_{\widetilde{F}})$ and some $u\in\U$ (e.g., $u=\kappa(x)$).
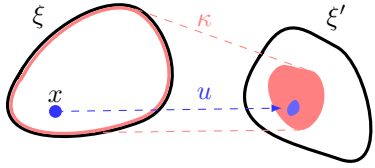
\begin{figure} 
\centering
\begin{tikzpicture}[scale=2.2]
	\draw [fill=none,draw=black, very thick] (0.9mm,5.4mm)
	.. controls ++(1.5mm,2.7mm) and ++(-1.9mm,1.0mm) .. ++(7.1mm,3.8mm)
	.. controls ++(1.9mm,-1.0mm) and ++(3.2mm,1.3mm) .. ++(0.2mm,-6.8mm)
	.. controls ++(-4.1mm,-1.6mm) and ++(-3.0mm,-4.4mm) ..  cycle
	;
	
	\draw [scale=0.95,fill=none,draw=red!50, very thick] (1.2mm,5.6mm)
	.. controls ++(1.5mm,2.7mm) and ++(-1.9mm,1.0mm) .. ++(7.12mm,3.9mm)
	.. controls ++(1.9mm,-1.0mm) and ++(3.2mm,1.3mm) .. ++(0.2mm,-6.8mm)
	.. controls ++(-4.1mm,-1.6mm) and ++(-3.2mm,-4.4mm) ..  cycle
	;
	
	\draw [fill=none,draw=black, very thick] (15.4mm,6.4mm)
	.. controls ++(1.5mm,2.7mm) and ++(-1.9mm,1.0mm) .. ++(4.6mm,0.5mm)
	.. controls ++(1.9mm,-1.0mm) and ++(3.2mm,-1.1mm) .. ++(0.2mm,-6.8mm)
	.. controls ++(-3.2mm,1.1mm) and ++(-3.0mm,-4.4mm) .. cycle
	;
	\draw [fill=red!50,draw=none] (16.0mm,5.2mm)
	.. controls ++(0.7mm,1.1mm) and ++(-1.1mm,0.1mm) .. ++(2.3mm,0.3mm)
	.. controls ++(1.1mm,-0.1mm) and ++(1.3mm,1.0mm) .. ++(0.1mm,-3.4mm)
	.. controls ++(-1.3mm,-1.1mm) and ++(-0.7mm,-1.1mm) .. cycle
	;
	\draw [fill=blue!60,draw=none] (17.0mm,3.1mm)
	.. controls ++(0.0mm,0.2mm) and ++(-0.2mm,0.0mm) .. ++(0.5mm,0.6mm)
	.. controls ++(0.2mm,-0.0mm) and ++(0.3mm,0.2mm) .. ++(0.1mm,-0.8mm)
	.. controls ++(-0.5mm,-0.4mm) and ++(-0.0mm,-0.2mm) ..  cycle
	;
	\draw[red!50,dashed] (0.3,0.17) -- (1.78,0.188);
	\draw[red!50,dashed] (0.78,0.92) -- (1.868,0.534);
	\fill[blue!80] (0.3,0.3) circle (1.1pt);
	\draw[dashed,-latex,blue!80] (0.3,0.3) -- (1.67,0.32);
	
	\node[below] at(0.2,1) {$\xi$};
	\node[below] at(2,1) {$\xi'$};
	\node[above] at(0.3,0.3)  {$x$};
	\node[below] at(1.2,0.5)  {\color{blue!80}$u$};
	\node at(1.2,0.85)  {\color{red!50}$\kappa$};
\end{tikzpicture}
\caption{{For {a transition} system ${\Sys}$ and {state-feedback abstraction} $\widetilde{\Sys}$, we represent $(\xi,x,\kappa,u) \in R^e$, i.e., the extended alternating simulation relation given in the proof of Lemma~\ref{lem:S_sim_tildeS}.}}\label{fig:altsim}
\end{figure}

Finally, let us present a final result that allows us to ensure that an upper bound on the cost of a trajectory of system $\Sys$ can be derived from its state-feedback abstraction. The following definition is in order.
\begin{definition}\label{def:lyapunov}
	A value function $v:\X\rightarrow\R$ is a \emph{Lyapunov-like function} with stage cost $\J(x,u)$ for a symbolic system $\Sys = (\mathcal{X}, \mathcal{U}, \trans{})$ in a domain $\X_v\subseteq\X$ if {$v(x)$ is bounded from below within $\X_v$ and  for all $x\in\X_v$ }there exists $u\in\Available(x,\trans{})$ fulfilling the Bellman inequality
	\begin{equation}\label{eq:bellman_inequality}
		v(x) \geq \J(x,u)+\max_{x'\in \Post(x,\trans{},u)}v(x').
	\end{equation}
	Moreover, if $v(x)$ satisfies {the non-strict inequality}~\eqref{eq:bellman_inequality} at the equality, its said to be the \emph{optimal cost-to-go function}.\envend
\end{definition}

The following theorem generalizes Theorem~8 in \cite{legat2021abstraction} to derive a Lyapunov-like function for an alternating simulation based on another Lyapunov-like function of its simulated system.
\begin{theorem}[{Generalized from \cite{legat2021abstraction}}]\label{theo:Lyapunov-like}
	{Consider transition systems} $\Sys_1 = (\mathcal{X}_1, \mathcal{U}_1, \trans{}_1)$ and $\Sys_2 = (\mathcal{X}_2, \mathcal{U}_2, \trans{}_2)$  such that $\Sys_2$ is an alternating simulation for $\Sys_1$, as given in Definition~\ref{def:altsim}. Let $R(x_2)\subseteq\X_1$ be the set of all $x_1\in\X_1$ such that $(x_1,x_2)$ are in the alternating simulation relation $R$. If $v_1:\X_1\rightarrow\R$ is a Lyapunov-like function for system $\Sys_1$ with cost function $\J_1(x_1,u_1)$ then 
	\begin{equation}
		v_2(x_2) = \min_{x_1\in R(x_2)}v_1(x_1)
	\end{equation}
	is a Lyapunov-like function for $\Sys_2$ with any cost function $\J_2(x_2,u_2)$ that verifies
	\begin{equation}\label{eq:costs_relation}
		\J_1(x_1,u_1)\geq \J_2(x_2,u_2),~~\forall (x_1,x_2,u_1,u_2)\in R^e.
	\end{equation}
\end{theorem}
\begin{proof}
	 From the definition of $v_2(x_2)$, for any $x_2\in\X_2$ we have $v_2(x_2)= v_1(x_1)$ for all $x_1\in R^*(x_2):=\arg\min_{x_1\in R(x_2)}v_1(x_1)$. 
	 	 
	 As $v_1(x_1)$ is a Lyapunov-like function, it satisfies the Bellman inequality~\eqref{eq:bellman_inequality} for all $x_1\in R^*(x_2)$ and some input $u_1\in\Available(x_1,\trans{}_1)$. Therefore,
	 \begin{align}
	  \J_1(x_1,u_1)+\max_{x_1'\in\Post(x_1,\trans{}_1,u_1) } v_1(x_1') &\leq v_1(x_1)\nonumber\\
	  &= v_2(x_2)\label{eq:cost_proof_step_1}
	 \end{align}
	
	Also because $x_1\in R^*(x_2)$ we have that $x_1$ and $x_2$ are in alternating simulation relation and, thus, for any $u_1\in\Available(x_1,\trans{}_1)$ (for instance, one satisfying the Bellman inequality) there exists an $u_2\in\Available(x_2,\trans{}_2)$ such that $(x_1,x_2,u_1,u_2)\in R^e$. As a consequence, we can use~\eqref{eq:costs_relation} and~\eqref{eq:cost_proof_step_1} to obtain
	\begin{equation}\label{eq:cost_proof_step_2}
		v_2(x_2) \geq  \J_2(x_2,u_2)+\max_{x_1'\in\Post(x_1,\trans{}_1,u_1)} v_1(x_1').
	\end{equation}

	By the definition of the extended alternating simulation relation $R^e$ we have that $(x_1,x_2,u_1,u_2)\in R^e$ implies that for any $x_2'\in\Post_2(x_2,\trans{}_2,u_2)$  the set $\Post(x_1,\trans{}_1,u_1)\cap R(x_2')$ is not empty, showing that
	\[\begin{aligned}
		v_2(x_2') &= \min_{x_1'\in R(x_2')} v_1(x_1')\\
		& \leq \min_{x_1'\in\Post_1(x_1,\trans{}_1,u_1)\cap R(x_2')} v_1(x_1')\\
		& \leq \max_{x_1'\in\Post_1(x_1,\trans{}_1,u_1)\cap R(x_2')} v_1(x_1')\\
		& \leq \max_{x_1'\in\Post_1(x_1,\trans{}_1,u_1)} v_1(x_1').
	\end{aligned}\]
	This, together with~\eqref{eq:cost_proof_step_2} yields 
	\begin{equation}
		v_2(x_2) \geq \J_2(x_2,u_2)+v_2(x_2') 
	\end{equation}
	for all $x_2'\in\Post_2(x_2,\trans{}_2,u_2)$ showing that $v_2(x_2)$ is a Lyapunov-like function for $\Sys_2$ with cost $\J_2(x_2,u_2)$.
\end{proof}

The improvement of Theorem~\ref{theo:Lyapunov-like} with respect to Theorem~8 in \cite{legat2021abstraction} is that it allows that more than one $x_1$ belongs to $R(x_2)$ whereas in \cite{legat2021abstraction} it was imposed that $R(x_2)$ is a singleton set. This is important in our context as, in practice, this enables the use of abstractions with overlapping cells to obtain Lyapunov-like functions for the systems they are in alternating simulation relation with. This aspect will be further discussed in the following section.

\section{Building symbolic abstractions with affine-feedback controllers}
\subsection{Transition design}
{An important classic result recalled in this section is the so-called S-procedure \cite[Section~2.6.3]{boyd1994linear}, given below.}
\begin{lemma}[S-procedure \cite{boyd1994linear}]\label{lem:Sproc}
Let $N+1$ quadratic functions $q_i(x)=x^\top R_ix+2s_i^\top x +r_i$,~$i\in\{0,\dots,N\}$. We have  $q_0(x)\geq 0$ for all $x\in\R^n$ such that $q_1(x)\geq0,\dots,q_N(x)\geq0$ if there exists $\beta_1\geq0,\dots,\beta_N\geq0$ such that
\begin{equation}
	\begin{bmatrix}
		R_0 & s_0\\
		s_0^\top & r_0
	\end{bmatrix}\succeq \sum_{i=1}^{N}\beta_i
	\begin{bmatrix}
		R_i & s_i\\
		s_i^\top & r_i
	\end{bmatrix}.
\end{equation} 
The converse holds if $N=1$.\envend
\end{lemma}

The present methodology for building the symbolic abstraction can be independently carried out over the $N_p$ subsystems of~\eqref{eq:pwaffinesys} and, to ease the notation, let us consider a single affine system
\begin{equation}
x(k+1) = Ax(k)+Bu(k)+g+\omega(k),~x(0)=x_0\label{eq:affinesys}
\end{equation}
where, at instant $k\in\N$,  $x(k)\in\X\subseteq\R^{n_x}$ is the state vector, $u(k)\in\U\subseteq \R^{n_u}$ is the control input, and $\omega(k)\in\Omega\coloneqq\co\{\omega_1,\dots,\omega_{N_\omega}\}\subset\R^{n_x}$ is a polytopic non-deterministic disturbance.	Let us consider a starting set 
\begin{equation}
	\B_s = \{x\in\R^n~:~ (x-c)'P(x-c)\leq1\}\label{eq:starting_set}
\end{equation}
 and a final set 
\begin{equation}
	\B_f = \{x\in\R^n~:~ (x-c_+)'P_+(x-c_+)\leq1\}, \label{eq:final_set}
\end{equation}
where the positive definite matrices $P,~P_+$, and the centers $c,~c_+,$ are known \textit{a priori} and chosen such that $\xi\subseteq\B_s$ and $\xi'\supseteq\B_f$ for a given pair of cells $(\xi,\xi')\in\X_d$. Our goal is to verify whether there exists a controller $\kappa$ such that $u=\kappa(x)$ ensures $x'\in\xi'$ for all $x\in\xi$ and $x'\in\Post(x,\trans{}_F,u)$. Also, for all $x\in\xi$, the controller $\kappa(x)$ must generate only valid inputs inside $\Available(x,\trans{}_F)\subseteq \U$.

In this paper we consider affine controllers of the form
\begin{equation}
\kappa(x)=K(x-c)+\ell\label{eq:control}
\end{equation}
where $K\in\R^{n_u\times n_x}$ and $\ell\in\R^{n_u}$ are designed to ensure that for all $x(k)\in\B_s$  we have that the control $u(k)=\kappa(x(k))\in\U$  implies that $x(k+1)\in\B_f$, under any $\omega(k)\in\Omega$. The choice of affine controllers is reasonable in our context as they can represent local first-order approximations of continuous non-linear ones.
Let us consider 
\begin{equation}
\U = \bigcap_{i=1}^{N_u} \bar\U_i,\qquad \bar\U_i=\{ u\in\R^{n_u}~:~ ||U_iu||_2 \leq 1 \},\label{eq:set_U}
\end{equation}
for given matrices $U_i$ of appropriate dimensions. Notice that $\U$ is the intersection of (possibly degenerated) ellipsoids.
\begin{theorem}\label{theo:th1}
For all $x(k)\in\B_s$ we have $x(k+1)\in\B_f$ for system~\eqref{eq:affinesys} under the constrained affine control law~\eqref{eq:control} and any $\omega(k)\in\Omega\coloneqq\co\{\omega_1,\dots,\omega_{N_\omega}\}\subset\R^{n_x}$ if and only if there exist $K\in\R^{n_u\times n_x}$, $\ell\in\R^{n_u}$, and scalars $\beta_i>0$, $\tau_i>0$ such that
\begin{equation}
	\begin{bmatrix}
		\beta_i P	& 0 		& (A+BK)^\top\\
		\bullet & 1-\beta_i	& \mu_i^\top\\
		\bullet & \bullet	& P_+^{-1}
	\end{bmatrix}\succeq0,~ i\in\{1,\dots,N_\omega\}\label{eq:lmi1}
\end{equation}
with $\mu_i= g+Ac+B\ell+\omega_i-c_+ $ and
\begin{equation}
	\begin{bmatrix}
		\tau_i P	& 0 		& K^\top U_i^\top\\
		\bullet & 1-\tau_i	& \ell^\top U_i^\top\\
		\bullet & \bullet	& I
	\end{bmatrix}\succeq0,\quad i\in\{1,\dots,N_u\},\label{eq:lmi2}
\end{equation}
{with given matrices $U_i$ defining $\U$ in~\eqref{eq:set_U}}.
\end{theorem}
\begin{proof}
Let $\tilde{x}=[x^\top~1]^\top$, consider an arbitrary $\omega\in\Omega$ and rewrite system~\eqref{eq:affinesys} under the control law~\eqref{eq:control} as
\[\tilde{x}(k+1) = \Ac\tilde{x}(k),~~\Ac \!=\! \begin{bmatrix}
	A+BK & g+B\ell-BKc+\omega\\ 0 &1
\end{bmatrix}\!.\] 
Similarly, define the matrices
\[\Pc = \begin{bmatrix}
	P & -Pc\\
	-c^\top P & c^\top Pc
\end{bmatrix},\quad \Pc_+ = \begin{bmatrix}
	P_+ & -P_+c_+\\
	-c_+^\top P_+ & c_+^\top P_+c_+
\end{bmatrix}\]
which can be used to equivalently redefine $\B_s$ as $\{x\in\R^n~:~\tilde{x}^\top \Pc\tilde{x}\leq1\}$ and $\B_f$ in the same fashion. Recalling the S-procedure (Lemma~\ref{lem:Sproc}), there exist $\beta>0$ fulfilling the inequality
\begin{equation}
	\Ac^\top \Pc_+\Ac-\Lc\preceq\beta ( \Pc-\Lc)  \label{eq:ineq_proof}
\end{equation}
with $\Lc=\diag(0,\dots,0,1)$, if and only if for all $x(k)\in\B_s$ we have $x(k+1)\in\B_f$. Now, let us show that~\eqref{eq:ineq_proof} is equivalent to~\eqref{eq:lmi1}. Developing the former, we obtain
\[
\begin{bmatrix}
	\Acl^\top P_+\Acl & \Acl^\top P_+\chi\\
	\bullet  		 & \chi^\top P_+\chi-1
\end{bmatrix}\preceq \beta\begin{bmatrix}
	P & -Pc\\
	\bullet & c^\top Pc-1
\end{bmatrix},
\]
with $\Acl=A+BK$ and $\chi = g+B\ell-BKc+\omega-c_+$,
which is rewritten by the Schur Complement Lemma \cite{boyd1994linear} as
\[\begin{bmatrix}
	\beta P 		& -\beta Pc 		& \Acl^\top \\
	\bullet & \beta (c^\top Pc-1) + 1		& \chi^\top \\
	\bullet & \bullet	& P_+^{-1}
\end{bmatrix}\succeq0\]
Then, multiply this last inequality to the right by
\[\Theta=\begin{bmatrix}
	I & c & 0\\
	0 & 1 & 0\\
	0 & 0 & I
\end{bmatrix}\] and to the left by its transpose to obtain
\begin{equation}
	\begin{bmatrix}
		\beta P	& 0 		& \Acl^\top \\
		\bullet & 1-\beta	& c^\top \Acl^\top +\chi^\top \\
		\bullet & \bullet	& P_+^{-1}
	\end{bmatrix}\succeq0.
\end{equation}
Given that $\omega\in\Omega=\co\{\omega_1,\dots,\omega_{N_\omega}\}$ we verify that this last inequality is equivalent to~\eqref{eq:lmi1} considering some $\beta\in\co\{\beta_1,\dots,\beta_{N_\omega}\}$.

Finally, to ensure that $\kappa(x)\in\U$ for all $x\in\B_s$ let us again recall the S-procedure, which allows rewriting this statement equivalently as: there exist scalars $\tau_i>0$ such that
\begin{equation}
	\begin{bmatrix}
		K^\top\! U_i^\top\! U_iK &\!\!\!\! -\!K^\top\! U_i^\top\! U_i(Kc\!-\!\ell)\\
		\bullet  & \!\!\!\!\!(Kc\!-\!\ell)^\top\!\! U_i^\top \!U_i(Kc\!-\!\ell)
	\end{bmatrix}\!-\!\Lc\!\preceq\tau_i ( \Pc\!-\!\Lc),\nonumber
\end{equation}
for all $i\in\{1,\dots,N_u\},$ which, after applying the Schur Complement Lemma, yields
\[\begin{bmatrix}
	\tau_i P 		& -\tau_i Pc 		& K^\top U_i^\top \\
	\bullet & \tau_i (c^\top Pc-1) + 1		&  -(Kc-\ell)^\top U_i^\top \\
	\bullet & \bullet	& I
\end{bmatrix}\succeq0.\]
Finally, multiplying this inequality to the right by $\Theta$ and to the left by its transpose yields~\eqref{eq:lmi2}, concluding the proof.
\end{proof}

\subsection{Optimal cost bound}

To introduce a performance objective to this control design, we consider a general quadratic transition cost function of the form
\begin{equation}\label{eq:cost}
	\J(x,u) = \begin{bmatrix}
		x\\
		u\\
		1
	\end{bmatrix}^\top \!\!Q\begin{bmatrix}
		x\\
		u\\
		1
	\end{bmatrix} 
\end{equation}
defined for some given positive semi-definite matrix $Q$ that can be decomposed as $Q=L^\top L$. Notice that this form is general enough to characterize, for instance, any quadratic cost defined as $\J(x,u)=(u-u_0)^\top Q_u(u-u_0)+ (x-x_0)^\top Q_x(x-x_0) $ for positive semi-definite  matrices $Q_u,Q_x$ and vectors $u_0,~x_0$ of appropriate dimensions, or a time-cost $\J(x,u)=T$ for some sampling-time $T>0$. Even more representative costs could be considered without major difficulties by adopting a transition-cost function defined as the maximum over several quadratic forms given as in~\eqref{eq:cost} but, for sake of clarity of the following developments, we decided to consider here the quadratic cost.

\begin{corollary}\label{coro:opt_prob_cost}
For given sets $\B_s$ and $\B_f$ defined in~\eqref{eq:starting_set}-\eqref{eq:final_set}, the solution to the convex-optimization problem
\begin{align}
	\inf_{K,\ell,\beta_i\geq0,\tau_i\geq0,\gamma\geq0,\widetilde{\J}\geq0} \widetilde{\J} \quad {\rm s.t.}~\text{\eqref{eq:lmi1}},& ~\text{\eqref{eq:lmi2},}\label{eq:opt_prob_obj}\\
	 \begin{bmatrix}
		\gamma P & 0 & [I~~K^\top ~~0]L^\top  \\
		\bullet &\widetilde{\J}-\gamma & [c^\top ~~\ell^\top ~~1]L^\top \\
		\bullet & \bullet & I
	\end{bmatrix}\succeq0 \label{eq:opt_prob_cnstr}
\end{align}
{satisfies $\widetilde{\J}=\max_{x\in\B_s}\J(x,\kappa(x))$ where the state-feedback control $\kappa(x)$ given in~\eqref{eq:control} ensures a transition from $\B_s$ to $\B_f$ as given in Theorem~\ref{theo:th1}}. 
\end{corollary}
\begin{proof}
	The proof that $\widetilde{\J}$ is an upper bound on the stage cost $\J(x,u)$ inside $\B_s$ under the control law~\eqref{eq:control} follows similar steps as the proof of Theorem~\ref{theo:th1}. First, considering~\eqref{eq:control}, one can rewrite 
	\begin{equation}
		\begin{bmatrix}
			x\\u\\1
		\end{bmatrix} = V\begin{bmatrix}
		x\\1
	\end{bmatrix},~~~V=\begin{bmatrix}
		I & 0\\
		K & \ell-Kc\\
		0 & 1
	\end{bmatrix},
	\end{equation}
	which allows to equivalently express the stage cost~\eqref{eq:cost} as $\J(x,\kappa(x))= \tilde{x}^\top V^\top L^\top LV\tilde{x}$.
	Now, consider the condition ``$\widetilde{\J}\geq\J(x,u)$ for all $x\in\B_s$''. Through the S-procedure, this is equivalent to the existence of $\gamma\geq0$ such that
	\begin{equation}
		\tilde{x}^\top V^\top L^\top LV\tilde{x} -\widetilde{\J}\leq \gamma(\tilde{x}^\top \Pc \tilde{x} -\Lc),
	\end{equation}
	which, in turn, can be rewritten through the Schur Complement Lemma as
	\begin{equation}
		\begin{bmatrix}
			\gamma P  & -\gamma Pc & \\
			\bullet & \gamma (c^\top Pc-1) + \widetilde{\J} & \smash{\raisebox{.5\normalbaselineskip}{$[V^\top \! L^\top]$}}\\ 
			\bullet &\bullet & I 
		\end{bmatrix}\succeq 0.
	\end{equation}
	Multiplying this inequality to the right by $\Theta$ and to the left by $\Theta^\top$ yields~\eqref{eq:opt_prob_cnstr}. The minimization over $\widetilde{\J}$ ensures $\max_{x\in\B_s}\J(x,\kappa(x))=\widetilde{\J}$, concluding the proof.
\end{proof}

The convex optimization problem  \eqref{eq:opt_prob_obj}-\eqref{eq:opt_prob_cnstr} is the core of our methodology. It not only provides an efficient way (by semi-definite programming) to determine whether there exists a controller~\eqref{eq:control} taking all states from $\B_s$ to $\B_f$ fulfilling the system requirements but also provides a tight upper-bound $\widetilde{\J}$ on the stage cost $\J(x,u)$. The tightness of this upper-bound is ensured by the fact that all mathematical steps employed in the proof of Theorem~\ref{theo:th1} and Corollary~\ref{coro:opt_prob_cost} preserves equivalence between statements (e.g., Schur Complement Lemma, S-procedure with $N=1$, congruence transformations).

 Finally, notice that the upper bound $\widetilde{\J}$ {is a cost function verifying the condition of Theorem~\ref{theo:Lyapunov-like}, allowing to use it} to determine a Lyapunov-like function for system~\eqref{eq:pwaffinesys} with cost~\eqref{eq:cost} when a Lyapunov-like function for the symbolic system $\widetilde{\Sys}$ with transition cost $\widetilde{\J}$ is provided.

Before presenting the related numerical experiments, let us present some further remarks regarding the conservatism of the presented approach. First, let us consider a nominal system with no input constraints and no exogenous input, i.e., $\U=\R^{n_u}$ and $\Omega=\{0\}$. In an optimistic scenario where there exist $\ell\in\R^{n_u}$ such that $Ac+B\ell+g=c_+$ (i.e., $c_+$ is reached from $c$ with $u=\ell$) we would have $\mu_i=0$  in the LMI~\eqref{eq:lmi1} and this constraint becomes simply a reminiscent of the Lyapunov inequality $(A+BK)^\top P_+(A+BK) -\beta P\preceq 0$ for some $\beta\in[0,1]$. If the starting and final sets $\B_s$ and $\B_f$ are chosen to have the same size and shape, we have $P_+=P$. Therefore, in this setting, a necessary and sufficient condition for the existence of a controller $\kappa$ as in~\eqref{eq:control} ensuring the transition from  $\B_s$ to $\B_f$ is that the pair $(A,B)$ is stabilizable. This provides us with a good guess for the shape of the sets $\B_s$ and $\B_f$, which can be defined using matrices $P$ satisfying the stabilizability condition. Also notice that, if a state-independent  control input $\kappa=\ell$ was considered, a necessary condition for feasibility of~\eqref{eq:lmi1} would be the stability of $A$, which is an overwhelming requirement from a control theory perspective. Indeed, the presence of the linear term on the state in the affine feedback controllers~\eqref{eq:control} is paramount to ensure that the state trajectories in the starting ball are not locally diverging from one another, reducing the non-determinism in the symbolic model.
One last observation is that the problem is linear in $c$, $c_+$ and $P_+^{-1}$. This can also allow the use of this method to adjust the positioning of the sets $\B_s$ and $\B_f$ or the shape of $P_+^{-1}$ in order to attain feasibility, to better bound the transition cost, or to search for least-violating controllers.

\section{Numerical experiments}
\subsection{One single transition}\label{sec:single_transition}
In this example we study aspects of determining a single transition for given starting and final sets $\B_s$ and  $\B_f$, as in~\eqref{eq:starting_set} and~\eqref{eq:final_set}. Consider the system \eqref{eq:affinesys} given by $A=e^{TA_c}$, $B=\int_{0}^{T}e^{tA_c}B_cdt$, $g=0$ with $T=0.5$,
\begin{equation}
	A_c=\begin{bmatrix}
	0 & 1 & 0\\
	0 & 0 & 1\\
	1 &-1 &-1
	\end{bmatrix},~B_c = \begin{bmatrix}
	0\\0\\1
\end{bmatrix}.
\end{equation}
The control input $u$ is constrained as $|u|\leq10$ and the exogenous input $\omega$ has entries bounded in norm by $\omega_{\max} = 0.01$.
Although unstable, this system is stabilizable and
\[P_0= \begin{bmatrix}
	2.8106 & 1.6583 & 1.0143\\
	1.6583 & 4.4629 & 1.4071\\
	1.0143 & 1.4071 & 2.3453
\end{bmatrix}\]
satisfies the closed-loop Lyapunov inequality $\Acl^\top P_0\Acl-P_0\prec0$ for some matrix $K$. Then, we define $\B_s$ and  $\B_f$ by $P=\nu^{-1} P_0$, $P_+=\eta P$, $c=0$ and $c_+= [0.1~0.5~1.9]^\top$ where $\nu>0$ is a \emph{volume multiplier} for $\B_s$ and $\eta>0$ is a \emph{contraction ratio} factor. In practice, greater values of $\nu$ imply larger volumes ${\rm vol}(\B_s)$ and  $\eta$ verifies the equality ${\rm vol}(\B_f)\eta^{{n_x}/2}={\rm vol}(\B_s)$. For several different values of $\eta$ and $\nu$, we solved the optimization problem \eqref{eq:opt_prob_obj}-\eqref{eq:opt_prob_cnstr} considering the quadratic cost $\J(x,u)=x^\top x +u^\top u$.

On average, each solution to \eqref{eq:opt_prob_obj}-\eqref{eq:opt_prob_cnstr} was found in 0.0158 seconds on an Intel$^{\circledR}$ Core$^{\text{\tiny TM}}$ i7-10610U CPU \@ 1.80~GHz$\times$8 with $16$~GB of memory and using the Julia JuMP \cite{JuMP2017} interface with the Mosek solver on Ubuntu 20.04.

In Figure~\ref{fig:ex1}, on the left-hand-side plot, one can notice that for increasing volumes of $\B_s$ or increasing contraction ratios $\eta$, larger costs are obtained. Indeed, having larger volumes means that the set on which $\J(x,u)$ must be bounded increases and higher contraction ratios $\eta$ demand a larger control effort. On the right-hand side, in turn, for a fixed $\eta=1$, we varied $\omega_{\max}$ to investigate the effect of the exogenous input on the closed-loop spectral radius $\rho(\Acl)$. Doing so, we noticed that allowing for larger uncertainties $\omega(k)$ tends to shrink the spectrum of $\Acl$ and the nominal system is controlled more aggressively (i.e, with a larger decay-rate). This can be also interpreted as the controller mapping the nominal system (e.g., $\omega(k)=0$) from $\B_s$ into a smaller set inside $\B_f$ to compensate for the larger uncertainties and ensure a successful transition.

\begin{figure}

	\centering
	\includegraphics[width=0.39\linewidth]{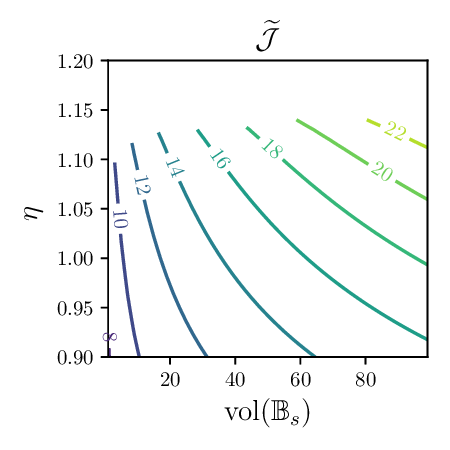}
	\includegraphics[width=0.39\linewidth]{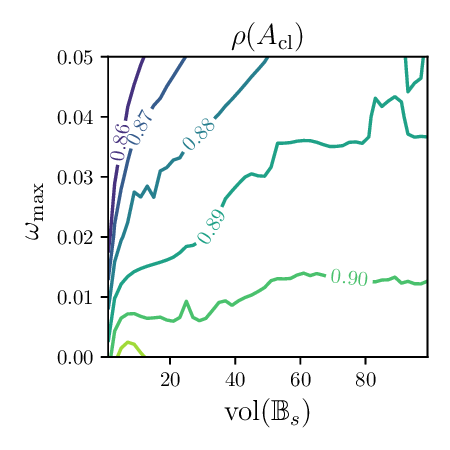}
	\caption{(Section~\ref{sec:single_transition}) To the left, value of the upper-bound $\widetilde{\J}$ on the cost function $\J(x,u)$ for different volumes of the starting set $\B_s$ and different contraction ratios $\eta$ of the final set $\B_f$. To the right, the closed-loop spectral radius $\rho(\Acl)$  for different volumes of the starting set $\B_s$ and varying bounds $\omega_{\max}$ on the exogenous disturbance $\omega(k)$. }\label{fig:ex1}
\end{figure}

\subsection{Optimal control}\label{ex:example_2}
In this example we provide one possible application of the method that uses the state-feedback transition system for the optimal control of piecewise-affine systems.
Consider a system $\Sys\coloneqq(\X, \U,\trans{ }_F)$ with the transition function $F(x,u)$ defined in~\eqref{eq:fun_Fxu} by 
\begin{equation}
	A_1=\begin{bmatrix}
		1.01 & 0.3\\
		-0.1 & 1.01
	\end{bmatrix}\nonumber,~B_1=\begin{bmatrix}
		1&0\\ 0 & 1
	\end{bmatrix},~g_1=\begin{bmatrix}
	-0.1\\-0.1
\end{bmatrix},
\end{equation}
$A_2=A_1^\top,~ A_3=A_1,~B_2=B_3=B_1,~g_2=0$ and $g_3=-g_1$. These systems are three spiral sources with unstable equilibria at $x_{e1}=[-0.9635~~0.3654]^\top,~x_{e2}=0,$ and $x_{e3}=-x_{e1}$. We also define the additive-noise sets $\Omega_1=\Omega_2=\Omega_3=[-0.05,0.05]^2$, the control-input set $\U=[-0.5,0.5]^2$ and the state space $\X=[-2,2]^2$. The $N_p=3$ partitions of $\X$ are $\X_1= \{x\in\X~:~x_1\leq-1 \},~\X_3= \{x\in\X~:~x_1>1 \},$ and $\X_2=\X\setminus(\X_1\cup\X_3)$. The goal is to bring the state $x$ from the initial set $\X_0$ to a final set $\X_*$, while avoiding the obstacle $\Oc$, presented in Figure~\ref{fig:ex2_traj_lyap}. The associated stage-cost function $\J(x,u)$ is defined in \eqref{eq:cost} with $Q=\diag(10^{-2}I,0)$ which evenly penalize states and inputs far away from the origin. To build a deterministic state-feedback abstraction in alternating simulation relation  with the system as described in Lemma~\ref{lem:S_sim_tildeS}, a set of balls of radius 0.2 covering the state space is adopted as cells $\xi\in\X_d$. We assume that inside cells intersecting the boundary of partitions of $\X$ the selected piecewise-affine mode is the same all over its interior and given by the mode defined at its center. An alternative to this is to split these cells and use the S-Procedure to incorporate the respective cuts into the design problem, but we do not proceed in this way to favor a clear illustration of the results. 

We thus compute state-feedback transitions between these cells using the results in Corollary~\ref{coro:opt_prob_cost}. To avoid solving the corresponding optimization problem for every pair of cells, we over approximate the reachable set of each cell by the growth-bound proposed in~\cite[Theorem~VIII.5]{reissig2016feedback} and only compute transition targeting cells with a non-empty intersection with this over-approximation for $\omega=0$. 

After about 206 seconds, 6984 state-feedback transitions were created among the cells, in the same computational setup as in the previous example. Finally, applying Dijkstra's algorithm \cite[p.~86]{bertsekas2012dynamic} to the reverse deterministic graph of the state-feedback abstraction with edge costs defined by the upper bounds $\widetilde{\J}$ (calculated from each solution of \eqref{eq:opt_prob_obj}-\eqref{eq:opt_prob_cnstr}), we obtained in 0.04 seconds a Lyapunov-like function for the abstraction, as given in Definition~\ref{def:lyapunov}. Given that the upper bounds $\widetilde{\J}$ satisfy the conditions of Theorem~\ref{theo:Lyapunov-like}, this Lyapunov-like function can be also used for the original system, taking into account the  underlying alternating simulation relation. A simulated trajectory implementing the controller associated to the shortest path found by Dijkstra's algorithm and undergoing random noise inputs is depicted in Figure~\ref{fig:ex2_traj_lyap}. For this problem, the guaranteed total cost was 2.05 (obtained from the Dijkstra's algorithm) whereas the true total cost of this {specific} trajectory was 1.29, showing that the \emph{worst-case} bound obtained was sufficiently close to {this value}. Also, the associated Lyapunov-like function is represented for each cell in the abstraction. {Although the methodology adopted in this example (discretizing the state-space) may suffer with the curse of dimensionality, its purpose is to illustrate the state-feedback transitions developed in this paper in optimal control problems with complex specifications.}

\begin{figure}
	\centering
	\begin{tikzpicture}[scale=0.8]
		\pgftext{%
			\includegraphics[width=0.5\linewidth]{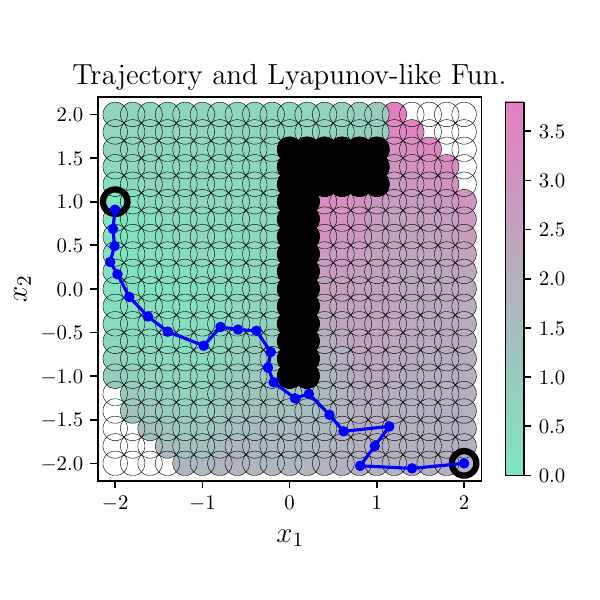}%
		}%
		\node at (0,1.2) {\color{white} \footnotesize$\Oc$};
		\node[fill=white, opacity=0.8, text opacity=1] at (-2.6,1.9) {\footnotesize$\X_*$};
		\node[fill=white, opacity=0.8, text opacity=1] at (2.4,-1.8) {\footnotesize$\X_0$};
	\end{tikzpicture}
\vspace{-.8cm}
	\caption{(Section~\ref{ex:example_2}) State trajectory (blue line) and Lyapunov-like function (color map) obtained for the optimal control problem of departing from $\X_0$ and reaching $\X_*$ under given specifications. Non-colored represents a region where no Lyapunov-like function could be defined in this state space.}\label{fig:ex2_traj_lyap}
\end{figure}

\section{Conclusions and Future Work}
In this work we propose \emph{state-feedback transitions} that are transitions between cells of a symbolic system parameterized by local feedback controllers that are correct-by-design, i.e., satisfy the transition requirements and input constraints. These transitions allows us to reduce the non-determinism in symbolic abstractions and to state an alternating simulation relation with the original system. The construction of a state-feedback transition is done by solving a convex optimization problem (with linear matrix inequalities as constraints) and a tight upper bound on the transition cost is then obtained. This allows us to compute Lyapunov-like functions for both the symbolic and real systems. Two examples illustrate particularities and usefulness of the presented methodology.

For future work, we seek to extend these results to design not only the transitions but also the positioning and shape of the cells along the optimal trajectories.

\bibliographystyle{ieeetr}
\bibliography{refs}
\end{document}